\documentclass[10pt,reqno]{amsart}

\usepackage{amscd}
\usepackage{amsfonts}
\usepackage{amsmath}
\usepackage{amssymb}
\usepackage{amsthm}
\usepackage{enumitem}
\usepackage{float}
\usepackage{mathtools}
\usepackage{mathrsfs}
\usepackage{needspace}
\usepackage{setspace}
\usepackage{tikz}
\usepackage{version}

\usepackage{hyperref}

\onehalfspace

\theoremstyle{plain}
\newtheorem{lemma}{Lemma}[section]

  \newtheorem{theorem}[lemma]{Theorem}
  \newtheorem*{theorem*}{Theorem}

  \newtheorem*{fact*}{Fact}
  \newtheorem*{claim*}{Claim}

  \newtheorem{conjecture}[lemma]{Conjecture}

\theoremstyle{definition}

\theoremstyle{remark}
  \newtheorem{remark}[lemma]{Remark}

\setenumerate{label=\textup{(}\roman*\textup{)}}

\makeatletter
\newsavebox{\@brx}
\newcommand{\llangle}[1][]{\savebox{\@brx}{\(\m@th{#1\langle}\)}%
  \mathopen{\copy\@brx\kern-0.5\wd\@brx\usebox{\@brx}}}
\newcommand{\rrangle}[1][]{\savebox{\@brx}{\(\m@th{#1\rangle}\)}%
  \mathclose{\copy\@brx\kern-0.5\wd\@brx\usebox{\@brx}}}
\makeatother

\title{Formulations of the PFR Conjecture over $\mathbb{Z}$}
\author{Freddie Manners}
\address{Freddie Manners, 450 Serra Mall Building 380, Stanford 94305, USA}
\email{fmanners@stanford.edu}

\begin{document}

\newcommand{\eps}[0]{\varepsilon}

\newcommand{\AAA}[0]{\mathbb{A}}
\newcommand{\CC}[0]{\mathbb{C}}
\newcommand{\EE}[0]{\mathbb{E}}
\newcommand{\FF}[0]{\mathbb{F}}
\newcommand{\NN}[0]{\mathbb{N}}
\newcommand{\PP}[0]{\mathbb{P}}
\newcommand{\QQ}[0]{\mathbb{Q}}
\newcommand{\RR}[0]{\mathbb{R}}
\newcommand{\TT}[0]{\mathbb{T}}
\newcommand{\ZZ}[0]{\mathbb{Z}}

\newcommand{\cA}[0]{\mathscr{A}}
\newcommand{\cB}[0]{\mathscr{B}}
\newcommand{\cC}[0]{\mathscr{C}}
\newcommand{\cD}[0]{\mathscr{D}}
\newcommand{\cE}[0]{\mathscr{E}}
\newcommand{\cF}[0]{\mathscr{F}}
\newcommand{\cH}[0]{\mathscr{H}}
\newcommand{\cG}[0]{\mathscr{G}}
\newcommand{\cK}[0]{\mathscr{K}}
\newcommand{\cL}[0]{\mathscr{L}}
\newcommand{\cM}[0]{\mathscr{M}}
\newcommand{\cN}[0]{\mathscr{N}}
\newcommand{\cP}[0]{\mathscr{P}}
\newcommand{\cR}[0]{\mathscr{S}}
\newcommand{\cS}[0]{\mathscr{S}}
\newcommand{\cT}[0]{\mathscr{S}}
\newcommand{\cU}[0]{\mathscr{U}}
\newcommand{\cW}[0]{\mathscr{W}}
\newcommand{\cX}[0]{\mathscr{X}}
\newcommand{\cY}[0]{\mathscr{Y}}
\newcommand{\cZ}[0]{\mathscr{Z}}

\newcommand{\fg}[0]{\mathfrak{g}}
\newcommand{\fk}[0]{\mathfrak{k}}
\newcommand{\fZ}[0]{\mathfrak{Z}}

\newcommand{\bmu}[0]{\boldsymbol\mu}

\newcommand{\AUT}[0]{\mathbf{Aut}}
\newcommand{\Aut}[0]{\operatorname{Aut}}
\newcommand{\Frob}[0]{\operatorname{Frob}}
\newcommand{\GI}[0]{\operatorname{GI}}
\newcommand{\HK}[0]{\operatorname{HK}}
\newcommand{\HOM}[0]{\mathbf{Hom}}
\newcommand{\Hom}[0]{\operatorname{Hom}}
\newcommand{\Ind}[0]{\operatorname{Ind}}
\newcommand{\Lip}[0]{\operatorname{Lip}}
\newcommand{\LHS}[0]{\operatorname{LHS}}
\newcommand{\RHS}[0]{\operatorname{RHS}}
\newcommand{\Sub}[0]{\operatorname{Sub}}
\newcommand{\id}[0]{\operatorname{id}}
\newcommand{\image}[0]{\operatorname{Im}}
\newcommand{\poly}[0]{\operatorname{poly}}
\newcommand{\trace}[0]{\operatorname{Tr}}
\newcommand{\sig}[0]{\ensuremath{\tilde{\cS}}}
\newcommand{\psig}[0]{\ensuremath{\cP\tilde{\cS}}}
\newcommand{\metap}[0]{\operatorname{Mp}}
\newcommand{\symp}[0]{\operatorname{Sp}}
\newcommand{\dist}[0]{\operatorname{dist}}
\newcommand{\stab}[0]{\operatorname{Stab}}
\newcommand{\HCF}[0]{\operatorname{hcf}}
\newcommand{\LCM}[0]{\operatorname{lcm}}
\newcommand{\SL}[0]{\operatorname{SL}}
\newcommand{\GL}[0]{\operatorname{GL}}
\newcommand{\rk}[0]{\operatorname{rk}}
\newcommand{\sgn}[0]{\operatorname{sgn}}
\newcommand{\uag}[0]{\operatorname{UAG}}
\newcommand{\freiman}[0]{Fre\u{\i}man}
\newcommand{\tf}[0]{\operatorname{tf}}
\newcommand{\ev}[0]{\operatorname{ev}}
\newcommand{\bg}[0]{\operatorname{big}}
\newcommand{\sml}[0]{\operatorname{sml}}
\newcommand{\vol}[0]{\operatorname{vol}}

\newcommand{\Conv}[0]{\mathop{\scalebox{1.5}{\raisebox{-0.2ex}{$\ast$}}}}
\newcommand{\bs}[0]{\backslash}
\newcommand{\heis}[3]{ \left(\begin{smallmatrix} 1 & \hfill #1 & \hfill #3 \\ 0 & \hfill 1 & \hfill #2 \\ 0 & \hfill 0 & \hfill 1 \end{smallmatrix}\right)  }
\newcommand{\uppar}[1]{\textup{(}#1\textup{)}}
\newcommand{\pol}[0]{\l}

\setcounter{tocdepth}{1}

\maketitle

\begin{abstract}
  The polynomial Fre\u{\i}man--Ruzsa conjecture is a fundamental open question in additive combinatorics.  However, over the integers (or more generally $\RR^d$ or $\ZZ^d$) the optimal formulation has not been fully pinned down.

  The conjecture states that a set of small doubling is controlled by a very structured set, with polynomial dependence of parameters.  The ambiguity concerns the class of structured sets needed.  A natural formulation in terms of generalized arithmetic progressions was recently disproved by Lovett and Regev.  A more permissive alternative is in terms of \emph{convex progressions}; this avoids the obstruction, but uses is a significantly larger class of objects, yielding a weaker statement.

  Here we give another formulation of PFR in terms of Euclidean ellipsiods (and some variations).  We show it is in fact equivalent to the convex progression version; i.e.~that the full range of convex progressions is not needed.  The key ingredient is a strong result from asymptotic convex geometry.
\end{abstract}

\section{Introduction}

A celebrated theorem of Fre\u{\i}man \cite{freiman} states that if $A \subseteq \ZZ$ is a finite set satisfying the small doubling hypothesis $|A+A| \le K |A|$ for some small $K$, where $A +A$ is the sumset $\{x+y \colon x,y \in A\}$, then $A$ must be contained in a generalized arithmetic progression, i.e.~a set of the form
\[
  P = \big\{ a_0 + n_1 a_1 + \dots + n_d a_d \colon n_i \in \{0,\dots,N_i - 1\} \big\}
\]
for some integers $a_i$, where the rank $d$ and size\footnote{Note this quantity may not be the same as the cardinality $|P|$ if the sums are not distinct.  Such a case is called an \emph{improper} generalized arithmetic progression.  We use the term ``size'' in this technical sense throughout.} $N_1 \dots N_d$ of the generalized arithmetic progression are bounded by functions of $K$ only.  An analogue for subsets of general abelian groups was obtained by Green and Ruzsa \cite{green-ruzsa}.

The polynomial Fre\u{\i}man--Ruzsa conjecture is a central open question in additive combinatorics, and asks for essentially optimal quantitative bounds in modified versions of these structural results.  The most commonly discussed case is when $A \subseteq \FF_p^n$ for some bounded $p$; then the conjecture states that if $|A+A| \le K |A|$ then $A$ is contained in $O(K^{O(1)})$ cosets of the same subgroup $H \le \FF_p^n$, where $|H| = O(K^{O(1)}) |A|$.

For subsets of $\ZZ$, or more generally $\ZZ^m$ or $\RR^m$, there has been less much consensus on what the correct statement should be.  One natural formulation is the following.
\begin{conjecture}[PFR; GAP formulation]
  \label{conj:gap-pfr}
  If $G = \RR^m$ or $\ZZ^m$ and $A \subseteq G$ is a finite set satisfying $|A+A| \le K |A|$, then there exists a generalized arithmetic progression
  \[
    P = \big\{ a_0 + n_1 a_1 + \dots + n_d a_d \colon n_i \in \{0,\dots,N_i - 1\} \big\}
  \]
for some choice of $a_i \in G$, with rank\footnote{Here and subsequently all $O(\cdot)$ constant are absolute and in particular independent of $m$.} $d = O(\log 2 K)$ and size $O(K^{O(1)}) |A|$; and a set $X \subseteq G$, $|X| = O(K^{O(1)})$; such that $A \subseteq P + X$.
\end{conjecture}
Unfortunately this formulation is false: this was shown recently by Lovett and Regev \cite{lovett-regev}, answering a question of Green \cite{green}.  Their counterexample (in $\RR^m$) has the form $A = B \cap \cL$ where $B \subseteq \RR^m$ is a large Euclidean ball and $\cL \subseteq \RR^m$ is a randomly chosen lattice of full rank.

This leaves open the following formulation, first discussed (in a closely related form) in \cite{green}.

\begin{conjecture}[PFR; convex formulation]
  \label{conj:convex-pfr}
  For $A$, $G$, $K$ as in Conjecture \ref{conj:gap-pfr}, there exists a \emph{convex progression}
  \[
    P = \big\{ a_0 + n_1 a_1 + \dots + n_d a_d \colon (n_1,\dots,n_d) \in \ZZ^d \cap B \big\}
  \]
where $B \subseteq \RR^d$ is some centrally symmetric convex body and $a_0,\dots,a_d \in G$ are given, with rank $d = O(\log 2 K)$ and size $|B \cap \ZZ^d| = O(K^{O(1)}) |A|$; and a set $X \subseteq G$, $|X| = O(K^{O(1)})$; such that $A \subseteq P + X$.
\end{conjecture}

The previous formulation is (equivalent to) the special case of this one where $B$ must be an axis-aligned cuboid $[-N_1,N_1] \times \dots \times [-N_d, N_d]$; and such $B$ do not suffice.  It is natural to ask how how large a collection of convex sets $B$ is necessary for the conjecture to have a chance of being true.

Our main formulation will use only convex sets $B$ which are (not necessarily axis-aligned) \emph{Euclidean ellipsiods}, i.e.~sets of the form $B = \{ v \in \RR^d \colon \|\gamma(v)\|_2 \le 1\}$ for some $\gamma \in \GL_d(\RR)$.  That is, we state the following:

\begin{conjecture}[PFR; ellipsoid formulation]
  \label{conj:ellipsoid-pfr}
  For $A$, $G$, $K$ as before, there exists an \emph{ellipsoid progression}
  \[
    P = \big\{ a_0 + n_1 a_1 + \dots + n_d a_d \colon \vec{n} = (n_1,\dots,n_d) \in \ZZ^d,\, \|\gamma(\vec{n})\|_2 \le 1 \big\}
  \]
  where $a_0,\dots,a_d \in G$ and $\gamma \in \GL_d(\RR)$ are given, with rank $d = O(\log 2 K)$ and size $|\{ \vec{n} \in \ZZ^d \colon \|\gamma(\vec{n})\|_2 \le 1 \}| = O(K^{O(1)}) |A|$; and a set $X \subseteq G$, $|X| = O(K^{O(1)})$; such that $A \subseteq P + X$.
\end{conjecture}

Again, the $P$ here are a special case of those in Conjecture \ref{conj:convex-pfr}.  Our main result is:

\begin{theorem}
  \label{thm:main}
  Conjectures \ref{conj:convex-pfr} and \ref{conj:ellipsoid-pfr} are equivalent.
\end{theorem}
I.e.~if Conjecture \ref{conj:convex-pfr} is true at all then it suffices to consider convex sets $B$ that are ellipsoids.

\begin{remark}
  In fact there is nothing special about ellipsoids: it is true, and our proof will implicitly show, that the class of all convex bodies $\{ \gamma(B_0) \colon \gamma \in \GL_d(\RR) \}$ is sufficient for any fixed convex body $B_0$ (or rather, one for each $d$).  For instance, yet another formulation would be in terms of \emph{skew progressions} (not a standard term)
  \[
    P = \big\{ a_0 + n_1 a_1 + \dots + n_d a_d \colon \vec{n} = (n_1,\dots,n_d) \in \ZZ^d,\, |\phi_1(\vec{n})|, \dots, |\phi_d(\vec{n})| \le 1 \big\}
  \]
for some $a_i$ and some basis of linear forms $\phi_1,\dots,\phi_d \colon \RR^d \to \RR$; i.e., taking $B_0 = [-1,1]^d$.  So, the weakness of Conjecture \ref{conj:gap-pfr} exploited by the Lovett--Regev argument is not the restriction on the shape of $B$, but the requirement that it be aligned to some lattice basis for $\ZZ^d$.
\end{remark}
\begin{remark}
  Another variant would be to replace the set $P$ by a Gaussian density
  \[
    \theta(x) = \sum_{\vec{n} \colon a_0 + n_1 a_1 + \dots + n_d a_d = x} \exp(-\|\gamma(\vec{n})\|_2^2)
  \]
and replace the covering requirement $A \subseteq X + P$ by a correlation one such as $\langle 1_A, \theta \rangle \gg K^{-O(1)} \|\theta\|_2 \|1_A\|_2$.  This is readily seen to be equivalent to Conjecture \ref{conj:ellipsoid-pfr} using standard tools.
\end{remark}

The non-trivial ingredient in the proof of Theorem \ref{thm:main} comes from asymptotic convex geometry, and can be encapsulated in the following (very much non-trivial) result due to Milman \cite{milman}.\footnote{The reader could consult \cite{gm} for an overview of these ideas.  In the case that $B_2$ is a Euclidean ball and $\gamma_1 = \id$, the ellipsoid $\gamma_2(B_2)$ is referred to as the \emph{$M$-ellipsoid} of $B_1$.}

\begin{theorem}[Milman's reverse Brunn--Minkowski inequality]
  \label{thm:mil}
  There is an absolute constant $c > 0$ such that the following holds.  For any $d$, and any two convex bodies $B_1$ and $B_2$ in $\RR^d$, there exist volume-preserving linear maps $\gamma_1, \gamma_2 \in \SL_d(\RR)$ such that for all $t_1, t_2 > 0$:
  \[
    \vol(t_1 \gamma_1(B_1) + t_2 \gamma_2(B_2))^{1/d} \le c \left(t_1 \vol(B_1)^{1/d} + t_2 \vol(B_2)^{1/d} \right) \, .
  \]
  It is clear one can take $\gamma_1 = \id$ if desired.
\end{theorem}

\section{Proof of the main theorem}

As we have stated, most of the work in the proof is done by Theorem \ref{thm:mil}.

\begin{proof}[Proof of Theorem \ref{thm:main}]
  Suppose $A \subseteq G$ with $|A+A| \le K |A|$ is given.  Applying Conjecture \ref{conj:convex-pfr} to $A$, we are given a symmetric convex body $C \subseteq \RR^d$, elements $a_0 \in G$, $\vec{a} \in G^d$ and $X \subseteq G$ such that $d = O(\log 2 K)$, $|C \cap \ZZ^d| = O(K^{O(1)}) |A|$, $|X| = O(K^{O(1)})$ and $A \subseteq P + X$ where
  \[
    P = \{ a_0 + \vec{a} \cdot \vec{n} \colon \vec{n} \in C \cap \ZZ^d \} \, .
  \]
Let $B_0$ denote the standard Euclidean ball $\{ v \in \RR^d \colon \|v\|_2 \le R\}$ where $R$ is chosen so that $\vol B_0 = \vol C = V$.  Applying Theorem \ref{thm:mil}, we obtain $\gamma \in \SL_d(\RR)$ such that,  writing $B$ for the ellipsoid $\gamma(B_0)$, we have
  \begin{equation}
    \label{eq:vol}
    \vol(t_1 C + t_2 B) \le c^d (t_1 + t_2)^d V
  \end{equation}
  for any $t_1, t_2 > 0$.  We make the following claim:
  \begin{claim*}
    There exist finite sets $Y, Z \subseteq \ZZ^d$ with $|Y|, |Z| = \exp(O(d))$, such that
    \[
      (C \cap \ZZ^d) \subseteq Y + (B \cap \ZZ^d)
    \]
    and
    \[
      (B \cap \ZZ^d) \subseteq Z + (C \cap \ZZ^d)\, .
    \]
  \end{claim*}
  Given this, we can deduce that
  \[
    |B \cap \ZZ^d| \le |Z|\, |C \cap \ZZ^d| = \exp(O(d)) O(K^{O(1)}) |A| = O(K^{O(1)}) |A|
  \]
  as $d = O(\log 2 K)$, and that $A \subseteq P + X \subseteq P' + X'$ where
  \[
    P' = \big\{ a_0 + \vec{a} \cdot \vec{n} \colon \vec{n} \in B \cap \ZZ^d \big\}
  \]
  and
  \[
    X' = X + \big\{ \vec{a} \cdot \vec{y} \colon \vec{y} \in Y \big\} \, ,
  \]
  meaning $|X'| \le |X|\, |Y| = O(K^{O(1)})$; so this suffices to proves the result.

  \begin{proof}[Proof of claim]
    This is a fairly standard packing/covering argument.  Let $Y$ be a maximal subset of $C \cap \ZZ^d$ such that the sets $y + B/2$ for $y \in Y$ are disjoint.  By maximality, $C \cap \ZZ^d \subseteq Y + B$, and hence
    \[
      C \cap \ZZ^d \subseteq (Y + B) \cap \ZZ^d = Y + (B \cap \ZZ^d) \, .
    \]
    Also, each set $y + B/2$ for $y \in Y$ is contained in $C + B/2$, so by disjointness and volume-counting we have
    \[
      |Y| \le \frac{\vol(C + B/2)}{\vol(B/2)} \le \frac{c^d (3/2)^d V}{(1/2)^{d} V} = (3 c)^d
    \]
    by \eqref{eq:vol}.  The argument for $Z$ is analogous, exchanging the roles of $B$ and $C$.
  \end{proof}
  This completes the proof of Theorem \ref{thm:main}.
\end{proof}

\bibliography{master}{}
\bibliographystyle{alpha}
\end{document}